\documentclass[10pt,twocolumn]{IEEEtran}

\usepackage{xcolor}

\usepackage{amsmath}
\usepackage{amsfonts}
\usepackage{amssymb}  

\usepackage{amsthm}
\theoremstyle{plain}
\newtheorem{thm}{Theorem}[section]
\newtheorem{lem}[thm]{Lemma}
\newtheorem{prop}[thm]{Proposition}
\newtheorem{cor}[thm]{Corollary}

\theoremstyle{definition}
\newtheorem{defn}{Definition}[section]

\newtheorem{exmp}{Example}[section]
\theoremstyle{remark}

\newcommand{\norm}[1]{\left\Vert #1 \right\Vert}
\newcommand{\abs}[1]{\left\vert #1 \right\vert}
\newcommand{\set}[1]{\left\{ #1 \right\}}
\newcommand{\dom}[1]{\mathrm{dom} \, #1 }
\newcommand{\tr}[1]{\mathrm{Tr} \left( #1 \right)}
\newcommand{\supp}[1]{\mathrm{supp}\, #1}
\newcommand{\para}{K}

\title{Sparsistency of $\ell_1$-Regularized $M$-Estimators}
\author{Yen-Huan~Li, Jonathan Scarlett, Pradeep~Ravikumar, and~Volkan~Cevher \thanks{This work was supported in part by the European Commission under Grant MIRG-268398, ERC Future Proof, SNF 200021-132548, SNF 200021-146750 and SNF CRSII2-147633.}
\thanks{Y.-H. Li, J. Scarlett, and V. Cevher are with the Laboratory for Information and Inference Systems, \'{E}cole Polytechnique F\'{e}d\'{e}rale de Lausanne (EPFL). P. Ravikumar is with the University of Texas at Austin.}}
\date{\empty}

\begin{document}

\maketitle

%

%

%
%
%

\begin{abstract}
We consider the model selection consistency or \emph{sparsistency}  of a broad set of $\ell_1$-regularized $M$-estimators for linear and non-linear statistical models in a unified fashion. 
For this purpose, we propose the local structured smoothness condition (LSSC) on the loss function. We provide a general result giving deterministic sufficient conditions for sparsistency in terms of the regularization parameter, ambient dimension, sparsity level, and number of measurements. We show that several important statistical models have $M$-estimators that indeed satisfy the LSSC, and as a result, the sparsistency guarantees for the corresponding $\ell_1$-regularized $M$-estimators can be derived as simple applications of our main theorem.
\end{abstract}

\section{Introduction} 
This paper studies the class of $\ell_1$-regularized $M$-estimators for \emph{sparse} high-dimensional estimation \cite{Buhlmann2011}. A key motivation for adopting such estimators is {sparse} model selection, that is, selecting the important subset of entries of a high-dimensional parameter based on random observations.  We study the conditions for the reliable recovery of the sparsity pattern, commonly known as model selection consistency or \emph{sparsistency}. 

For the specific case of sparse linear regression, the $\ell_1$-regularized least squares estimator has received considerable attention. With respect to sparsistency,
results have been obtained for both the noiseless case  (e.g.,~\cite{CandesTao05,Donoho06,DonTan05}) and the noisy case~\cite{Meinshausen06,Wainwright2009,Zhao2006}. While sparsistency results have been obtained for $\ell_1$-regularized $M$-estimators on some \emph{specific} non-linear models such as logistic regression and Gaussian Markov random field models \cite{Bach2010,Bunea08,LamFan07,Meinshausen06,Ravikumar2010,Ravikumar2011}, \emph{general} techniques with broad applicability are largely lacking.


%
%

Performing a general sparsistency analysis requires the identification of general properties of statistical models, and their corresponding $M$-estimators, that can be exploited to obtain strong performance guarantees. In this paper, we introduce the \emph{local structured smoothness condition} (LSSC) condition (Definition \ref{def_TBD_condition}), which controls the smoothness of the objective function in a particular structured set. We illustrate how the LSSC enables us to address a broad set of sparsistency results in a unified fashion, including logistic regression, gamma regression, and graph selection. We explicitly check the LSSC for these statistical models, and as in previous works \cite{Fan2011,Fan2004,Ravikumar2010,Ravikumar2011,Wainwright2009,Zhao2006}, we derive sample complexity bounds for the high-dimensional setting, where the ambient dimension and sparsity level are allowed to scale with the number of samples.

To the best of our knowledge, the first work to study the sparsistency of a broad class of models was that of \cite{Fan2011} for generalized linear models; however, the technical assumptions therein appear to be difficult to check for specific models, thus making their application difficult.  Another related work is \cite{Lee2014}; in Section \ref{sec_discussions}, we compare the two, and discuss a key advantage of our approach.


The paper is organized as follows. We specify the problem setup in Section \ref{sec_setup}. We introduce the LSSC in Section \ref{sec_TBD}, and give several examples of functions satisfying the LSSC in Section \ref{sec_examples}. In Section \ref{sec_deterministic_condition}, we present the main theorem of this paper, namely, sufficient conditions for an $\ell_1$-regularized $M$-estimator to successfully recover the support. Sparsistency results for four different statistical models are established in Section \ref{sec_appl} as corollaries of our main result. In Section \ref{sec_discussions}, we present further discussions of our results, and list some directions for future research.  The proofs of our results can be found in the appendix.

\section{Problem Setup} \label{sec_setup}

We consider a general statistical modeling setting where we are given $n$ independent samples $\{y_{i}\}_{i=1}^{n}$ drawn from some distribution $\mathbb{P}$ with a sparse parameter $\beta^* := \beta(\mathbb{P}) \in \mathbb{R}^p$ that has at most $s$ non-zero entries. We are interested in estimating this sparse parameter $\beta^*$ given the $n$ samples via an $\ell_1$-regularized $M$-estimator of the form
\begin{equation}
\hat{\beta}_n := \arg \min_{\beta \in \mathbb{R}^p} L_n( \beta ) + \tau_n \left\Vert \beta \right\Vert_1, \label{eq_betahat}
\end{equation}
where $L_n$ is some convex function, and $\tau_n > 0$ is a regularization parameter.

We mention here a special case of this model that has broad applications in machine learning. For fixed vectors $x_1, \ldots, x_n$ in $\mathbb{R}^p$, suppose that we are given realizations $y_1, \ldots, y_n$ of independent random variables $Y_1, \ldots, Y_n$ in $\mathbb{R}$.  We assume that each $Y_i$ follows a probability distribution $P_{\theta_i}$ parametrized only by $\theta_i$, where $\theta_i := \left\langle x_i, \beta^* \right\rangle$ for some sparse parameter  $\beta^* \in \mathbb{R}^p$. Then it is natural to consider the $\ell_1$-regularized maximum-likelihood estimator
\begin{equation}
\hat{\beta}_n := \arg \min_{\beta \in \mathbb{R}^p} \frac{1}{n} \sum_{i = 1}^n \ell ( y_i; \beta, x_i ) + \tau_n \left\Vert \beta \right\Vert_1, \notag
\end{equation}
where $\ell$ denotes the negative log-likelihood at $y_i$ given $x_i$ and $\beta$. Thus, we obtain (\ref{eq_betahat}) with $L_n( \beta ) := \frac{1}{n} \sum_{i = 1}^n \ell( y_i; \beta, x_i )$. 

There are of course many other examples; to name one other, we mention the graphical learning problem, where we want to learn a sparse concentration matrix of a vector-valued random variable.  In this setting, we also arrive at the formulation (\ref{eq_betahat}), where $L_n$ is the negative log-likelihood of the data~\cite{Ravikumar2011}.

We focus on the \emph{sparsistency} of $\hat{\beta}_n$; roughly speaking, an estimator $\hat{\beta}_n$ is sparsistent if it recovers the support of $\beta^*$ with high probability when the number of samples $n$ is large enough.

\begin{defn}[Sparsistency] \label{def_sparsistency}
A sequence of estimators $\{\hat{\beta}_n\}_{n=1}^{\infty}$ is called \emph{sparsistent} if 
\begin{equation}
\lim_{n \to \infty} \mathsf{P}\, \left\{ \mathrm{supp}\, \hat{\beta}_n \neq \mathrm{supp}\, \beta^* \right\} = 0. \notag
\end{equation}
\end{defn}

The main result of this paper is that, if the function $L$ is convex and satisfies the LSSC, and certain assumptions analogous to those used for linear models (see \cite{Wainwright2009}) hold true, then the $\ell_1$-regularized $M$-estimator $\hat{\beta}_n$ in (\ref{eq_betahat}) is sparsistent under suitable conditions on the regularization parameter $\tau_n$ and the triplet $(p,n,s)$.  We allow for the case of diverging dimensions \cite{Fan2011,Ravikumar2010,Ravikumar2011,Wainwright2009,Zhao2006}, where $p$ grows exponentially with $n$.

\subsection*{Notations and Basic Definitions}
Fix $v \in \mathbb{R}^p$, and let $\mathcal{P} = \{ 1, \ldots, p \}$. For any $\mathcal{S} \subseteq \mathcal{P}$, the notation $v_{\mathcal{S}}$ denotes the sub-vector of $v$ on $\mathcal{S}$, and the notation $v_{\mathcal{S}^\mathrm{c}}$ denotes the sub-vector $v_{\mathcal{P} \setminus \mathcal{S}}$. For $i \in \mathcal{P}$, the notation $v_i$ denotes $v_{\{ i \}}$. We denote the support set of $v$ by $\mathrm{supp}\, v$, defined as $\mathrm{supp}\, v = \{ i: v_i \neq 0, i \in \mathcal{P} \}$. The notation $\mathrm{sign}\, v$ denotes the vector $( \mathrm{sign}\, v_1, \ldots, \mathrm{sign}\, v_p )$, where $\mathrm{sign}\, v_i = v_i \left\vert v_i \right\vert^{-1}$ if $v_i \neq 0$, and $\mathrm{sign}\, v_i = 0$ otherwise, for all $i \in \mathcal{P}$. We denote the transpose of $v$ by $v^T$, and the $\ell_q$-norm of $v$ by $\left\Vert v \right\Vert_q$ for $q \in [1, + \infty]$. For $u, v \in \mathbb{R}^p$, the notation $\left\langle u, v \right\rangle$ denotes $\sum_{i = 1}^p u_i v_i$. 

For $A \in \mathbb{R}^{p \times p}$, the notations $A_{\mathcal{S}, \mathcal{S}}$, $A_{\mathcal{S}^c, \mathcal{S}}$, $\mathrm{supp}\, A$, $\mathrm{sign}\, A$, and $A^T$ are defined analogously to the vector case. The notation $\left\Vert A \right\Vert_q$ denotes the operator norm induced by the vector $\ell_q$-norm; in particular, $\norm{A}_2$ denotes the spectral norm of $A$.

Let $X$ be a real-valued random variable. We denote the expectation and variance of $X$ by $\mathsf{E}\, X$ and $\mathsf{var}\, X$, respectively. The probability of an event $\mathcal{E}$ is denoted by $\mathsf{P}\, \mathcal{E}$.

Let $f$ be a vector-valued function with domain $\dom{f} \subseteq \mathbb{R}^p$. The notations $\nabla f$ and $\nabla^2 f$ denote the gradient and Hessian mapping of $f$, respectively. The notation $f \in \mathcal{C}^k ( \dom{f} )$ means that $f$ is $k$-times continuously differentiable on $\dom{f}$. For a given function $f \in \mathcal{C}^k ( \dom{f} )$, its $k$-th order Fr\'{e}chet derivative at $x \in \dom{f}$ is denoted by $D^k f ( x )$, which is a multilinear symmetric form \cite{Zeidler1995}. The following special cases summarize how to compute all of the quantities related to the Fr\'{e}chet derivative in this paper:
\begin{enumerate}
\item The first order Fr\'{e}chet derivative is simply the gradient mapping; therefore, $D f ( x ) [ u ] = \left\langle \nabla f (x), u \right\rangle$ for all $u \in \mathbb{R}^p$.
\item The second order Fr\'{e}chet derivative is the Hessian mapping; therefore, $D^2 f ( x ) [ u, v ] = \left\langle u, \nabla^2 f (x) v \right\rangle$ for all $u, v \in \mathbb{R}^p$.
\item The third order Fr\'{e}chet derivative is defined as follows. We first define the 2-linear form (matrix) $D^3 f (x) [ u ] := \lim_{t \to 0} \frac{\nabla^2 f ( x + t u ) - \nabla^2 f (x)}{t}$. Then 
\begin{align}
D^3 f (x) [ u, v, w ] &= \left( D^3 f (x) [u] \right) [ v, w ] \notag \\
& = \left\langle v, (D^3 f (x) [u]) w \right\rangle. \notag
\end{align}
We then define the 1-linear form (vector) $D^3 f (x) [ u, v ]$ to be the unique vector such that $\langle D^3 f (x) [ u, v ], w \rangle = D^3 f (x) [ u, v, w ]$ for all vectors $w$ in $\mathbb{R}^p$. 
\item When the arguments are the same, we simply have $D^k f ( x ) [ u, \ldots, u ] = \left. \frac{d^k\phi_u(t)}{dt^k}\right\vert_{t = 0}$, where $\phi_u ( t ) := f ( x + t u )$.
\end{enumerate}

\section{Local Structured Smoothness Condition} \label{sec_TBD} 
The following definition provides the key property of convex functions that will be exploited in the subsequent sparsistency analysis. 

\begin{defn} [Local Structured Smoothness Condition (LSSC)] \label{def_TBD_condition}
Consider a function $f \in C^3 ( \mathrm{\dom{f}} )$ with domain $\dom{f} \subseteq \mathbb{R}^p$. Fix $x^* \in \dom{f}$, and let $\mathcal{N}_{x^*}$ be an open  set in $\dom{f}$ containing $x^*$. The function $f$ satisfies the \emph{$( x^*, \mathcal{N}_{x^*} )$}-LSSC with parameter $\para \geq 0$ if
\begin{equation}
\norm{ D^3 f ( x^* + \delta ) [ u, u ] }_{\infty} \leq \para \norm{ u }_2^2, \notag
\end{equation}
for all $\delta \in \mathbb{R}^p$ such that $x^* + \delta \in \mathcal{N}_{x^*}$, and for all $u \in \mathbb{R}^p$ such that $u_{\mathcal{S}^c} = 0$, where $\mathcal{S} := \supp{ x^* }$.
\end{defn}

Note that $D^3 f ( x^* + \delta ) [ u, u ]$ is a $1$-linear form, so $\|\cdot\|_{\infty}$ in Definition \ref{def_TBD_condition} is the vector $\ell_\infty$-norm.  The following equivalent characterization follows immediately.

\begin{prop} \label{prop_equivalence_TBD}
The function $f$ satisfies the \emph{$( x^*, \mathcal{N}_{x^*} )$}-LSSC with parameter $\para \geq 0$ if and only if
\begin{equation}
\abs{ D^3 f ( x^* + \delta ) [ u, u, e_j ] } \leq \para \norm{ u }_2^2, \label{eq:equiv}
\end{equation}
for all $\delta \in \mathbb{R}^p$ such that $x^* + \delta \in \mathcal{N}_{x^*}$, for all $u \in \mathbb{R}^p$ such that $u_{\mathcal{S}^c} = 0$, where $\mathcal{S} := \supp{ x^* }$, and for all $j \in \set{ 1, \ldots, p }$, where $e_j$ is the standard basis vector with $1$ in the $j$-th position and $0$s elsewhere.
\end{prop}

As we will see in the next section, this equivalent characterization is useful when verifying the LSSC for a given $M$-estimator.

Since differentiation is a linear operator, the LSSC is preserved under linear combinations with positive coefficients, as is stated formally in the following lemma.

\begin{lem} \label{lem_preservation_TBD} 
Let $f_1$ satisfy the $( x, \mathcal{N}_{1} )$-LSSC with parameter $\para_1$, and $f_2$ satisfy the $( x, \mathcal{N}_{2} )$-LSSC with parameter $\para_2$. Let $\alpha$ and $\beta$ be two positive real numbers. The function $f := \alpha f_1 + \beta f_2$ satisfies the $( x, \mathcal{N}_x )$-LSSC with parameter $\para$, where $\mathcal{N}_x := \mathcal{N}_1 \cap \mathcal{N}_2$, and $\para := \alpha \para_1 + \beta \para_2$.
\end{lem}

We conclude this section by briefly discussing the connection of the LSSC with other conditions.  The following result, Proposition 9.1.1 of \cite{Nesterov1994}, will be useful here and throughout the paper.

\begin{prop} \label{prop_911}
    Let $A$ be a 3-linear symmetric form on $\left( \mathbb{R}^p \right)^3$, and $B$ be a positive-semidefinite 2-linear symmetric form on $\left( \mathbb{R}^p \right)^2$. If
    \begin{equation}
        \abs{ A [ u, u, u ] } \leq B [ u, u ]^{3/2} \notag
    \end{equation}
    for all $u \in \mathbb{R}^p$, then
    \begin{equation}
        \abs{ A [ u, v, w ] } \leq B [ u, u ]^{1/2} B [ v, v ]^{1/2} B [ w, w ]^{1/2} \notag
    \end{equation}
    for all $u, v, w \in \mathbb{R}^p$.
\end{prop}

This proposition shows that the condition in (\ref{eq:equiv}) \emph{without structural constraints on $u$ and $e_j$} is equivalent to the statement that
\begin{equation}
    \abs{ D^3 f ( x^* + \delta ) [ u, v, w ] }  \leq K \norm{ u }_2 \norm{v}_2 \norm{w}_2 \label{eq_strong__LSS}
\end{equation}
for all $u, v, w \in \mathbb{R}^p$.  In the appendix, we show that (\ref{eq_strong__LSS}) holds for all $\delta \in \mathbb{R}^p$ such that $x^* + \delta \in \mathcal{N}_{x^*}$ if and only if
\begin{equation}
    \norm{ D^2 f ( x^* + \delta ) - D^2 f ( x^* ) }_2 \leq K \norm{ \delta }_2, \label{eq_RSS}
\end{equation}
for all $\delta \in \mathbb{R}^p$ such that $x^* + \delta \in \mathcal{N}_{x^*}$.  The latter condition is simply the local Lipschitz continuity of the Hessian of $f$. This is why we consider our condition a \emph{local structured smoothness} condition, with structural constraints on the inputs of the $D^3 f ( x^* + \delta )$ operator. 

The preceding observations reveal that (\ref{eq_strong__LSS}), or the equivalent formulation (\ref{eq_RSS}), is more restrictive than the LSSC.  That is, (\ref{eq_strong__LSS}) implies the LSSC, while the reverse is not true in general.

\section{Examples} \label{sec_examples}
In this section, we provide some examples of functions that satisfy the LSSC.



\begin{exmp} \label{exmp_linear}
Suppose that $f ( \beta ) := \norm{ y - X \beta }_2^2$ for some fixed $y \in \mathbb{R}^p$ and $X \in \mathbb{R}^{n \times p}$. Since $D^3 f ( \beta ) \equiv 0$ everywhere, the function $f$ satisfies the $( \beta^*, \mathcal{N}_{\beta^*} )$-LSSC with parameter $\para = 0$ for any $\beta^* \in \mathbb{R}^p$ and any open  set $\mathcal{N}_{\beta^*} \subseteq \mathbb{R}^p$ that contains $\beta^*$. This function appears in the negative-likelihood in the Gaussian regression model.
\end{exmp}

\begin{exmp} \label{exmp_loglinear}

Let $f ( \beta ) := \left\langle x, \beta \right\rangle - \ln \left\langle x, \beta \right\rangle$ for some fixed $x \in \mathbb{R}^p$. We show that, for any fixed $\beta^* \in \dom{f}$ such that $\beta^*_{\mathcal{S}^c} = 0$, there exists some non-negative $\para$ and some open set $\mathcal{N}_{\beta^*}$ such that $f$ satisfies the $( \beta^*, \mathcal{N}_{\beta^*} )$-LSSC with parameter $K$. This function appears in the negative log-likelihood in gamma regression with the canonical link function.

By a direct differentiation, we obtain for all $u \in \mathbb{R}^p$ that
\begin{align}
& \abs{ D^3 f ( \beta^* + \delta ) [ u, u, u ]  } \notag \\
&\quad = 2 \left( 1 + \gamma \right)^{-3} \left\{ D^2 f ( \beta^* ) [ u, u ] \right\}^{ 3 / 2  },\ \label{eq:gamma_ex_init}
\end{align}
where
\begin{equation}
\gamma := \frac{\left\langle x, \delta \right\rangle}{ \left\langle x, \beta^* \right\rangle}, \notag
\end{equation}
Combining this with Proposition \ref{prop_911}, we have for each standard basis vector $e_j$ that
\begin{align}
& \abs{ D^3 f ( \beta^* + \delta ) [ u, u, e_j ]  } \notag \\
&\quad \leq 2 \left( 1 + \gamma \right)^{-3} D^2 f ( \beta^* ) [ u, u ] \left\{ D^2 f ( \beta^* ) [ e_j, e_j ] \right\}^{ 1 / 2  } \notag \\
&\quad \leq 2 \left( 1 - \abs{ \gamma } \right)^{-3} D^2 f ( \beta^* ) [ u, u ] \left\{ D^2 f ( \beta^* ) [ e_j, e_j ] \right\}^{ 1 / 2  }. \notag 
\end{align}
Now define $\mathcal{S} := \supp{ \beta^* }$, and suppose that $u_{\mathcal{S}^c} = \delta_{\mathcal{S}^c} = 0$, and that 
\begin{equation}
\norm{ \delta }_2 \leq \frac{\left\langle x, \beta^* \right\rangle}{ ( 1 + \kappa ) \norm{ x_{\mathcal{S}} }_2 } \notag
\end{equation}
for some $\kappa > 0$. By the Cauchy-Schwartz inequality, it immediately follows that $\abs{ \gamma } \leq ( 1 + \kappa )^{-1} < 1$, and thus $\beta^* + \delta$ is in $\dom{f}$.  Moreover, using this bound on $|\gamma|$, we can further upper bound $|D^3 f|$ as
$$ \abs{ D^3 f ( \beta^* + \delta ) [ u, u, e_j ]  } \leq 2 \left( 1 + \kappa^{-1} \right)^3 \lambda_{\max} d_{\max}^{1/2} \norm{ u }_2^2 \notag, $$
where $\lambda_{\max}$ is the maximum restricted eigenvalue of $D^2 f( \beta^* )$ defined as
\begin{equation}
\lambda_{\max} := \sup_{ \norm{ u }_2 \leq 1 \atop u_{\mathcal{S}^c} = 0 } D^2 f ( \beta^* ) [ u, u ], \notag
\end{equation}
and $d_{\max}$ denotes the maximum diagonal entry of $\nabla^2 f ( \beta^* )$. Therefore, $f$ satisfies the $( \beta^*, \mathcal{N}_{\beta^*} )$-LSSC with parameter $\para := 2 ( 1 + \kappa^{-1} )^3 \lambda_{\max} d_{\max}^{1/2}$, where
\begin{equation}
\mathcal{N}_{\beta^*} := \set{ \beta^* + \delta: \norm{ \delta }_2 \leq \frac{\left\langle x, \beta^* \right\rangle}{ ( 1 + \kappa ) \norm{ x_{\mathcal{S}} }_2 }, \delta \in \mathbb{R}^p }. \notag
\end{equation}

\end{exmp}

\begin{exmp} \label{exmp_logdet}
Consider the function $f ( \Theta ) = \tr{ X \Theta } - \ln \det \Theta$ with a fixed $X \in \mathbb{R}^{p \times p}$, and with $\dom{f} := \set{ \Theta \in \mathbb{R}^{p \times p}: \Theta > 0 }$. We show that, for any fixed $\Theta^* \in \dom{f}$, there exists some non-negative $K$ and some open set $\mathcal{N}_{\Theta^*}$ such that $f$ satisfies the $( \Theta^*, \mathcal{N}_{\Theta^*} )$-LSSC with parameter $\para$. This function appears as the negative log-likelihood in the Gaussian graphical learning problem.

Note that the previous definitions (in particular, Definition \ref{def_TBD_condition}), should be interpreted here as being taken with respect to the vectorizations of the relevant matrices.

It is already known that $f$ is standard self-concordant \cite{Nesterov2004}; that is, 
\begin{equation}
\abs{ D^3 f ( \Theta^* + \Delta ) [ U, U, U ] } \leq 2 \left\{ D^2 f ( \Theta^* + \Delta ) [ U, U ] \right\}^{3/2}, \notag
\end{equation}
for all $U \in \mathbb{R}^{p \times p}$ and all $\Delta \in \mathbb{R}^{p \times p}$ such that $\Theta^* + \Delta \in \dom{f}$. This implies, by Proposition \ref{prop_911},
\begin{align}
\abs{ D^3 f ( \Theta^* + \Delta ) [ U, U, V ] } \leq 2 & \left\{ D^2 f ( \Theta^* + \Delta ) [ U, U ] \right\} \notag \\
& \left\{ D^2 f ( \Theta^* + \Delta ) [ V, V ] \right\}^{1/2}, \notag
\end{align}
for all $U, V \in \mathbb{R}^{p \times p}$, and all $\Delta \in \mathbb{R}^{p \times p}$ such that $\Theta^* + \Delta \in \dom{f}$.

Moreover, by a direct differentiation, 
\begin{align}
\norm{ D^2 f ( \Theta^* + \Delta ) }_2 &= \norm{ ( \Theta^* + \Delta )^{-1} \otimes ( \Theta^* + \Delta )^{-1} }_2 \notag \\
&= \norm{ \left( \Theta^* + \Delta \right)^{-1} }_2^2. \notag
\end{align}
Fix a positive constant $\kappa$, and suppose that we choose $\Delta$ such that $\norm{ \Delta }_F \leq ( 1 + \kappa )^{-1} \rho_{\min}$, where $\rho_{\min}$ denotes the smallest eigenvalue of $\Theta^*$. Since $\norm{ \Delta }_2 \leq \norm{ \Delta }_F$, it follows that $\norm{ \Delta }_2 \leq ( 1 + \kappa )^{-1} \rho_{\min}$, and, by Weyl's theorem \cite{Horn1985}, 
\begin{equation}
\norm{ \left( \Theta^* + \Delta \right)^{-1} }_2 \geq \frac{\kappa}{1 + \kappa} \rho_{\min}. \notag
\end{equation}

Combining the preceding observations, it follows that $f$ satisfies the $( \Theta^*, \mathcal{N}_{\Theta^*} )$-LSSC with parameter $K := 2 \kappa^{-3} ( 1 + \kappa )^3 \rho_{\min}^{-3}$, where
\begin{align}
\mathcal{N}_{\Theta^*} = \Big\{ \Theta^* + \Delta: \left\Vert \Delta \right\Vert_F < \frac{1}{1 + \kappa} \rho_{\min},  \qquad\notag \\
\Delta = \Delta^T, \Delta \in \mathbb{R}^{p \times p} \Big\}. \notag
\end{align}
Here we have not exploited the special structure of $U$ in Definition \ref{def_TBD_condition} (namely, $u_{\mathcal{S}^{\mathrm{c}}} = 0$), though conceivably the constant $K$ could improve by doing so.  Note that $\mathcal{N}_{\Theta^*} \subset \dom{f}$ and $\mathcal{N}_{\Theta^*}$ is convex.

\end{exmp}


\section{Deterministic Sufficient Conditions} \label{sec_deterministic_condition}  
We are now in a position to state the main result of this paper, whose proof can be found in the appendix.

Let $\beta^* \in \mathbb{R}^p$ be the true parameter, and let $\mathcal{S} = \left\{ i: ( \beta^* )_i \neq 0 \right\}$ be its support set. Define the ``genie-aided" estimator with exact support information:
\begin{equation}
\check{\beta}_n := \arg \min_{\beta \in \mathbb{R}^p: \beta_{\mathcal{S}^{\mathrm{c}}} = 0} L_n( \beta ) + \tau_n \left\Vert \beta \right\Vert_1, \label{eq_def_betacheck}
\end{equation}
where here and subsequently we assume that the $\arg\min$ is uniquely achieved. 

\begin{thm} \label{thm_main_deterministic}
Suppose that $\check{\beta}_n$ is uniquely defined. Then the $\ell_1$-regularized estimator $\hat{\beta}_n$ defined in (\ref{eq_betahat}) uniquely exists, successfully recovers the sign pattern, i.e., $\mathrm{sign}\, \hat{\beta}_n = \mathrm{sign}\, \beta^*$, and satisfies the error bound
\begin{equation}
\left\Vert \hat{\beta}_n - \beta^* \right\Vert_2 \leq r_n := \frac{\alpha + 4}{\lambda_{\min}} \sqrt{s} \tau_n, \label{eq_defn_Rn}
\end{equation}
if the following conditions hold true.

\begin{enumerate}
\item \emph{(Local structured smoothness condition)} $L_n$ is convex, three times continuously differentiable, and satisfies the $( \beta^*, \mathcal{N}_{\beta^*} )$-LSSC with parameter $K \geq 0$, for some convex $\mathcal{N}_{\beta^*} \subseteq \mathrm{dom}\, L_n$.

\item \emph{(Positive definite restricted Hessian)} The restricted Hessian at $\beta^*$ satisfies $\left[ \nabla^2 L_n( \beta^* ) \right]_{\mathcal{S}, \mathcal{S}} \geq \lambda_{\min} I$ for some $\lambda_{\min} > 0$.
 
\item \emph{(Irrepresentablility condition)} For some $\alpha \in (0, 1]$, it holds that 
\begin{equation}
\left\Vert \left[ \nabla^2 L_n( \beta^* ) \right]_{\mathcal{S}^{\mathrm{c}}, \mathcal{S}} \left[ \nabla^2 L_n( \beta^* ) \right]_{\mathcal{S}, \mathcal{S}}^{-1} \right\Vert_{\infty} < 1 - \alpha. \label{eq_irrepresentability}
\end{equation}

\item \emph{(Beta-min condition)} The smallest non-zero entry of $\beta$ satisfies
\begin{equation}
\beta_{\min} := \min \left\{ \left\vert ( \beta^*)_k \right\vert: k \in \mathcal{S} \right\} > r_n,\label{eq_beta_min}
\end{equation}
where $r_n$ is defined in (\ref{eq_defn_Rn}).

\item The regularization parameter $\tau_n$ satisfies
\begin{equation}
\tau_n < \frac{\lambda_{\min}^2}{4 \left( \alpha + 4 \right)^2} \frac{\alpha}{K s}. \label{eq_tau_n}
\end{equation}

\item The gradient of $L_n$ at $\beta^*$ satisfies
\begin{equation}
\left\Vert \nabla L_n ( \beta^* ) \right\Vert_\infty \leq \frac{\alpha}{4} \tau_n. \label{eq_gradient_bound}
\end{equation}

\item The relation $\mathcal{B}_{r_n} \subseteq \mathcal{N}_{\beta^*}$ holds, where
\begin{align}
\mathcal{B}_{r_n} &:= \left\{ \beta \in \mathbb{R}^p  : \left\Vert \beta_n - \beta^* \right\Vert_2 \leq r_n, \beta_{\mathcal{S}^c} = 0 \right\} \notag
\end{align}
and $r_n$ is defined in (\ref{eq_defn_Rn}).

\end{enumerate}


\end{thm}


As mentioned previously, the first condition is the key assumption permitting us to perform a general analysis.  The second, third, and forth assumptions are analogous to those appearing in the literature for sparse linear regression. We refer to \cite{Buhlmann2011} for a systematic discussion of these conditions.\footnote{Equation (\ref{eq_irrepresentability}) is sometimes called the \emph{incoherence condition} \cite{Wainwright2009}.}  

The remaining conditions determine the interplay between $\tau_n$ , $n$, $p$, and $s$.  Whether the relation $\mathcal{B}_{r_n} \subseteq \mathcal{N}_{\beta^*}$ holds depends on the specific $\mathcal{N}_{\beta^*}$ that one can derive for the given loss function $L_n$.  Whether the upper bound on $\left\Vert \nabla L_n( \beta^* ) \right\Vert_{\infty}$ holds depends on the concentration of measure behavior of $\nabla L_n( \beta^* )$, which usually concentrates around $0$. In the next section, we will give concrete examples for the high-dimensional setting, where $p$ and $s$ scale with $n$.

Of course, $\mathrm{sign}\, \hat{\beta}_n = \mathrm{sign}\, \beta^*$ implies that $\mathrm{supp}\, \hat{\beta}_n = \mathrm{supp}\, \beta^*$, i.e.~successful support recovery.

\section{Applications} \label{sec_appl}
In this section, we provide several applications of Theorem \ref{thm_main_deterministic}, presenting concrete bounds on the sample complexity in each case.  We defer the full proofs of the results in this section to the appendix.  However, in each case, we present here the most important step of the proof, namely, verifying the LSSC.

Note that instead of the classical setting where only the sample size $n$ increases, we consider the high-dimensional setting, where the ambient dimension $p$ and the sparsity level $s$ are allowed to grow with $n$ \cite{Fan2011,Fan2004,Ravikumar2010,Ravikumar2011,Wainwright2009,Zhao2006}.


\subsection{Linear Regression}
\label{SecGaussianRegression}
We first consider the linear regression model with additive sub-Gaussian noise. This setting trivially fits into our theoretical framework.

\begin{defn}[Sub-Gaussian Random Variables] \label{def:sub_gaussian}
    A zero-mean real-valued random variable $Z$ is \emph{sub-Gaussian} with parameter $c > 0$ if 
    \begin{equation}
    \mathsf{E}\, \exp( tZ ) \leq \exp \left( \frac{c^2 t^2}{2} \right) \notag
    \end{equation}
    for all $t \in \mathbb{R}$.
\end{defn}

Let $\mathcal{X}_n := \left\{ x_1, \ldots, x_n \right\} \subset \mathbb{R}^n$ be given. Define the matrix $X_n \in \mathbb{R}^{n \times p}$ such that the $i$-th row of $X_n$ is $x_i$. We assume that the elements in $\mathcal{X}_n$ are normalized such that each column of $X$ has $\ell_2$-norm less than or equal to $\sqrt{n}$. Let $W_1, \ldots, W_n$ be independent~sub-Gaussian random variables with parameter $c$, and define $Y_i := \left\langle x_i, \beta^* \right\rangle + W_i$.

We consider the $\ell_1$-regularized $M$-estimator of the form (\ref{eq_betahat}), with
\begin{equation}
L_n(\beta) = \frac{1}{n} \sum_{i = 1}^n \frac{1}{2}\left( Y_i - \left\langle x_i, \beta \right\rangle \right)^2. \notag
\end{equation}
As shown in the first example of Section \ref{sec_examples}, $L_n$ satisfies the LSSC with parameter $K = 0$ everywhere in $\mathbb{R}^p$. Therefore, the condition on $\tau_n$ in (\ref{eq_tau_n}) is trivially satisfied, as is the final condition listed in the theorem.

By a direct calculation, we have
\begin{equation}
\nabla L_n( \beta^* ) = \frac{1}{n} \sum_{i = 1}^n ( Y_i - \mathsf{E}\, Y_i ) x_i. \notag
\end{equation}
By the union bound and the standard concentration inequality for sub-Gaussian random variables \cite{Boucheron2013}, 
\begin{align}
&\mathsf{P}\, \left\{ \left\Vert \nabla L_n ( \beta^* ) \right\Vert_{\infty} \geq \frac{\alpha \tau_n}{4} \right\} \notag \\
&\quad\leq \sum_{i = 1}^p \mathsf{P}\, \left\{ \left\vert \left[ \nabla L_n ( \beta^* ) \right]_i \right\vert \geq \frac{\alpha \tau_n}{4} \right\} \notag \\
&\quad\leq \left. 2 p \exp \left( - c n t^2 \right) \right\vert_{t = \frac{\alpha \tau_n}{4}}. \notag
\end{align}

Since $[D^2 L_n( \beta )]_{\mathcal{S}, \mathcal{S}} = [D^2 L_n( \beta^* )]_{\mathcal{S}, \mathcal{S}}$ is positive definite for all $\beta \in \mathbb{R}^p$ by the second assumption of Theorem \ref{thm_main_deterministic}, $\check{\beta}_n$ uniquely exists, and Theorem \ref{thm_main_deterministic} is applicable. By choosing $\tau_n$ sufficiently large that the above bound decays to zero, we obtain the following.

\begin{cor}
For the linear regression problem described above, suppose that assumptions 2 to 4 of Theorem \ref{thm_main_deterministic} hold for some $\lambda_{\min}$ and $\alpha$ bounded away from zero.\footnote{For all of the examples in this section, these assumptions are independent of the data, and we can thus talk about them being satisfied \emph{deterministically}.} If $s \log p \ll n$, and we choose $\tau_n \gg ( n^{-1} \log p )^{1/2}$, then the $\ell_1$-regularized maximum likelihood estimator is sparsistent.
\end{cor}

Observe that this recovers the scaling law given in \cite{Wainwright2009} for the linear regression model.

\subsection{Logistic Regression}
Let $\mathcal{X}_n := \left\{ x_1, \ldots, x_n \right\} \subset \mathbb{R}^n$ be given. As in Section~\ref{SecGaussianRegression}, we assume that $\sum_{j=1}^{n}(x_i)_j^2 \le n$ for all $i\in\{1,\dotsc,p\}$.

Let $\beta^* \in \mathbb{R}^p$ be sparse, and define $\mathcal{S} := \mathrm{supp}\, \beta^*$. We are interested in estimating $\beta^*$ given $\mathcal{X}_n$ and $\mathcal{Y}_n := \left\{ y_1, \ldots, y_n \right\}$, where each $y_i$ is the realization of a Bernoulli random variable $Y_i$ with 
\begin{equation}
\mathsf{P}\, \left\{ Y_i = 1 \right\} = 1 - \mathsf{P}\, \left\{ Y_i = 0 \right\} = \frac{1}{1 + \exp \left( - \left\langle x_i, \beta^* \right\rangle \right)}. \notag
\end{equation}
The random variables $Y_1, \ldots, Y_n$ are assumed to be independent.

We consider the $\ell_1$-regularized maximum-likelihood estimator of the form (\ref{eq_betahat}) with
\begin{equation}
L_n ( \beta ) := \frac{1}{n} \sum_{i = 1}^n \ln \left\{ 1 + \exp \left[ - ( 2 Y_i - 1 ) \left\langle x_i, \beta \right\rangle \right] \right\}. \notag
\end{equation}

Define $\ell_i(\beta) = \ln \left[ 1 + \exp \left( - ( 2 y_i - 1 ) \left\langle x_i, \beta \right\rangle \right) \right]$.  The cases $y_i=0$ and $y_i=1$ are handled similarly, so we focus on the latter.  A direct differentiation yields the following (this is most easily verified for $u=v$):
\begin{align}
&|D^3 \ell_i ( \beta^* + \delta ) [ u, u, v ]| \notag \\
&= \frac{\left\vert 1 - \exp\left( - \left\langle x_i, \beta^* + \delta \right\rangle \right) \right\vert}{1 + \exp \left( - \left\langle x_i, \beta^* + \delta \right\rangle \right)} \left\vert \left\langle x_i, v \right\rangle \right\vert D^2 \ell_i ( \beta^* + \delta ) [ u, u ] \notag \\
&\leq \left\vert \left\langle x_i, v \right\rangle \right\vert D^2 \ell_i ( \beta^* + \delta ) [ u, u ], \notag
\end{align}
and
\begin{align}
D^2 \ell_i ( \beta ) [ u, u ] & = \frac{\exp \left( - \left\langle x_i, \beta \right\rangle \right) \left\langle x_i, u \right\rangle^2}{\left[ 1 + \exp \left( - \left\langle x_i, \beta \right\rangle \right) \right]^2} \notag \\
& \leq \frac{1}{4} \left\langle x_i, u \right\rangle^2 \notag
\end{align}
for all $\beta \in \mathbb{R}^p$. The last inequality follows since the function $\frac{z}{(1+z)^2}$ has a maximum value of $\frac{1}{4}$ for $z\ge0$.  It follows that
\begin{align}
|D^3 \ell_i ( \beta^* + \delta ) [ u, u, v ]| & \leq \frac{1}{4} \left\vert \left\langle x_i, v \right\rangle \right\vert \left\vert \left\langle x_i, u \right\rangle \right\vert^2 \notag \\
& \leq \frac{1}{4} \left\Vert ( x_i )_{\mathcal{S}} \right\Vert_2^2 \norm{ x_i }_{\infty} \left\Vert u \right\Vert_2^3, \notag
\end{align}
for any $u \in \mathbb{R}^p$ such that $u_{\mathcal{S}^c} = 0$, and for any $v$ equal to some standard basis vector $e_j$. Hence, $L_n$ satisfies the $( \beta^*, \mathcal{N}_{\beta^*} )$-LSSC with parameter $K = ( 1 / 4 ) \nu_n^2 \gamma_n$, where 
\begin{align}
\nu_n &:= \max_i \left\Vert ( x_i )_{\mathcal{S}} \right\Vert_2, \notag \\
\gamma_n &:= \max_i \| x_i \|_{\infty} , \notag
\end{align}
and where $\mathcal{N}_{\beta^*}$ can be any fixed open convex neighborhood of $\beta^*$ in $\mathbb{R}^p$. 

%

\begin{cor} \label{cor_logistic}
For the logistic regression problem described above, suppose that assumptions 2 to 4 of Theorem \ref{thm_main_deterministic} hold for some $\lambda_{\min}$ and $\alpha$ bounded away from zero.  If we choose $\tau_n \gg ( n^{-1} \log p )^{1/2}$,  and $s$ and $p$ such that $s^2 \left( \log p \right) \nu_n^4 \gamma_n^2 \ll n$, then the $\ell_1$-regularized maximum-likelihood estimator is sparsistent.
\end{cor}

In \cite{Bunea08}, a scaling law of the form $s \ll \frac{\sqrt n}{(\log n)^2}$ is given, but the result is restricted to the case that $p$ grows polynomially with $n$. The result in \cite{Bach2010} yields the scaling $s^2 (\log p) \overline{\nu_n}^2 \ll n$, where $\overline{\nu}_n := \max \set{ \norm{ x_i }_2 }$.  It should be noted that $\overline{\nu}_n$ is generally significantly larger than $\nu_n$ and $\gamma_n$; for example, for i.i.d.~Gaussian vectors, these scale on average as $O(\sqrt{p})$, $O(\sqrt{s})$ and $O(1)$, respectively.  Our result recovers the same dependence of $n$ on $s$ and $p$ as that in \cite{Bach2010}, but removes the dependence on $\overline{\nu}_n$. Of course, we do not restrict $p$ to grow polynomially with $n$.

\subsection{Gamma Regression}
Let $\mathcal{X}_n := \left\{ x_1, \ldots, x_n \right\} \subset \mathbb{R}^n$ be given. We again assume that $\sum_{j=1}^{n}(x_i)_j^2 \le n$ for all $i\in\{1,\dotsc,p\}$.

Let $\beta^* \in \mathbb{R}^p$ be sparse, and define $\mathcal{S} := \mathrm{supp}\, \beta^*$. We are interested in estimating $\beta^*$ given $\mathcal{X}_n$ and $\mathcal{Y}_n := \left\{ y_1, \ldots, y_n \right\}$, where each $y_i$ is the realization of a gamma random variable $Y_i$ with known shape parameter $k > 0$ and unknown scale parameter $\theta_i = k^{-1} \left\langle x_i, \beta^* \right\rangle^{-1}$.  The corresponding density function is
of the form $ \frac{1}{\Gamma(k)\theta_i^k}y_i^{k-1}e^{-\frac{y_i}{\theta_i}}$.

We assume that 
\begin{equation}
\left\langle x_i, \beta^* \right\rangle \geq \mu_n \quad  \forall i \in \{ 1, \ldots, n \} \label{eq:gamma_assump}
\end{equation}
for some $\mu_n > 0$, so $\theta_i$ is always well-defined. Moreover, the random variables $Y_1, \ldots, Y_n$ are assumed to be independent.

We consider the $\ell_1$-regularized maximum-likelihood estimator of the form (\ref{eq_betahat}) with
\begin{equation}
L_n( \beta ) := \frac{1}{n} \sum_{i = 1}^n \left[ - \ln \left\langle x_i, \beta \right\rangle + Y_i \left\langle x_i, \beta \right\rangle \right]. \notag
\end{equation}
Note that $\theta_i$ only enters the log-likelihood via constant terms not containing $\beta$; these have been omitted, as they do not affect the estimation.

Defining $\ell_i(\beta) = - \ln \left\langle x_i, \beta \right\rangle + y_i \left\langle x_i, \beta \right\rangle$, we obtain the following for all $u \in \mathbb{R}^p$ such that $u_{\mathcal{S}^c} = 0$, using the Cauchy-Schwartz inequality and \eqref{eq:gamma_assump}:
\begin{align}
D^2 \ell_i ( \beta^* ) [ u, u ] = \frac{\left\langle x_i, u \right\rangle^2}{\left\langle x_i, \beta^* \right\rangle^2} & \leq \frac{\left\Vert ( x_i )_{\mathcal{S}} \right\Vert_2^2}{\left\langle x_i, \beta^* \right\rangle^2} \left\Vert u \right\Vert_2^2 \notag \\
& \leq \frac{1}{\mu_n^2} \left\Vert u \right\Vert_2^2 \left\Vert ( x_i )_{\mathcal{S}} \right\Vert_2^2. \notag
\end{align}
Thus, the largest restricted eigenvalue of $D^2 \ell_i ( \beta^* )$ is upper bounded by $\mu_n^{-2}\nu_n^2$, where $\nu_n = \max_i \left\{ \left\Vert ( x_i )_{\mathcal{S}} \right\Vert_2 \right\}$. Similarly, we obtain
\begin{equation}
D^2 \ell_i ( \beta^* ) [ e_j, e_j ] \leq \frac{1}{\mu_n^2} \norm{ x_i }_{\infty}^2, \notag
\end{equation}
for any standard basis vector $e_j$. Thus, the largest diagonal entry of $D^2 \ell_i ( \beta^* )$ is upper bounded by $\mu_n^{-2} \gamma_n^2$, where $\gamma_n = \max_i \|x_i\|_{\infty}$.

Fix $\kappa > 0$. By Example \ref{exmp_loglinear}, $L_n$ satisfies the $( \beta^*, \mathcal{N}_{\beta^*} )$-LSSC with parameter $K = 2( 1 + \kappa^{-1} )^3 \mu_n^{-3}\nu_n^2 \gamma_n$, and 
\begin{equation}
\mathcal{N}_{\beta^*} = \left\{ \beta^* + \delta: \left\Vert \delta \right\Vert_2 < \frac{\mu_n}{(1 + \kappa)\nu_n}, \delta \in \mathbb{R}^p \right\}. \notag
\end{equation}



\begin{cor} \label{cor_gamma}
Consider the gamma regression problem as described above, and suppose that assumptions 2 to 4 of Theorem \ref{thm_main_deterministic} hold for some $\lambda_{\min}$, and $\alpha$ bounded away from zero. If $\tau_n \gg \sqrt{n}^{-1} \log p$ and $s^2 \left( \log p \right)^2 \mu_n^{-6}\nu_n^4 \gamma_n^2 \ll n$, then the $\ell_1$-regularized maximum likelihood estimator is sparsistent.
\end{cor}

To the best of our knowledge, this is the first sparsistency result for gamma regression.


\subsection{Graphical Model Learning} \label{subsec_graph_learning}
Let $\Theta^* \in \mathbb{R}^{p \times p}$ be a positive-definite matrix. We assume there are at most $s$ non-zero entries in $\Theta^*$, and let $\mathcal{S}$ denote its support set. Let $X_1, \ldots, X_n$ be independent $p$-dimensional random vectors generated according to a common distribution with mean zero and covariance matrix $\Sigma^* := \left( \Theta^* \right)^{-1}$.  We are interested in recovering the support of $\Theta^*$ given $X_1, \ldots, X_n$.


We assume that each $\left( \Sigma_{i,i} \right)^{-1/2}X_{i,i}$ is sub-Gaussian with parameter $c > 0$, and that $\Sigma_{i,i}$ is bounded above by a constant $\kappa_{\Sigma^*}$, for all $i \in \{ 1, \ldots, p \}$.  Let $\rho_{\min}$ denote the smallest eigenvalue of $\Theta^*$. 

We consider the $\ell_1$-regularized $M$-estimator of the form (\ref{eq_betahat}), given by
\begin{align}
\hat{\Theta}_n := \arg \min_{\Theta} \left\{ L_n ( \Theta ) + \tau_n \left\vert \Theta \right\vert_{1} : \Theta > 0, \Theta \in \mathbb{R}^{p \times p} \right\}. \notag
\end{align}
Here $\left\vert \Theta \right\vert_1$ denotes the entry-wise $\ell_1$-norm, i.e., $\left\vert \Theta \right\vert_1 = \sum_{(i,j) \in \{ 1, \ldots, p \}^2} \left\vert \Theta_{i,j} \right\vert$ and
\begin{equation}
L_n(\Theta) = \mathrm{Tr}\, \left( \hat{\Sigma}_n \Theta \right) - \log \det \Theta, \notag
\end{equation}
where $\hat{\Sigma}_n := \frac{1}{n} \sum_{i = 1}^n X_i X_i^T$ is the sample covariance matrix.

Fix $\kappa > 0$. By Example \ref{exmp_logdet}, we know that $L_n$ satisfies the $( \Theta^*, \mathcal{N}_{\Theta^*} )$-LSSC with parameter $2 \kappa^{-3} ( 1 + \kappa )^3 \rho_{\min}^{-3}$, where 
\begin{align}
\mathcal{N}_{\Theta^*} := \left\{ \Theta^* + \Delta: \left\Vert \Delta \right\Vert_F < \frac{1}{1 + \kappa} \rho_{\min}, \right. \notag \\
\left. \Delta = \Delta^T, \Delta \in \mathbb{R}^{p \times p} \right\}, \notag
\end{align}
where $\rho_{\min}$ denotes the smallest eigenvalue of $\Theta^*$.  

The beta-min condition can be written as
\begin{align}
\min \left\{ \Theta^*_{i,j}: \Theta^*_{i,j} \neq 0, ( i, j ) \in \{ 1, \ldots, p \}^2 \right\} \notag > r_n. \notag
\end{align}
We now have the following.

\begin{cor} \label{cor_graph_learning}
Consider the graphical model selection problem described above, and suppose the above assumptions and assumptions 2 to 4 of Theorem \ref{thm_main_deterministic} hold for some $c$, $\kappa_{\Sigma^*}$, $\rho_{\min}$, $\lambda_{\min}$, and $\alpha$ bounded away from zero.   If $\tau_n \gg ( n^{-1} \log p )^{1/2}$ and $s^2 \log p \ll n$, the $\ell_1$-regularized $M$-estimator $\hat{\Theta}_n$ is sparsistent.
\end{cor}

Corollary \ref{cor_graph_learning} is for graphical learning on general sparse networks, as we only put a constraint on $s$.  Several previous works have instead imposed structural constraints on the maximum degree of each node; e.g.~see~\cite{Ravikumar2011}. Since this model requires additional structural assumptions beyond sparsity alone, it is outside the scope of our theoretical framework.

\section{Discussion} \label{sec_discussions}

Our work bears some resemblance to the independent work of \cite{Lee2014}.  The smoothness condition therein is in fact the \emph{non-structured} condition in \eqref{eq_RSS}.  From the discussion in Section \ref{sec_TBD}, we see that our condition is less restrictive.
As a consequence, both analyses lead to scaling laws of the form $n \gg K^2 s^2 \log p$ for generalized linear models, but the corresponding definitions of $K$ differ significantly.  Eliminating the dependence of $K$ on $p$ requires additional non-trivial extensions of the framework in \cite{Lee2014}, whereas in our framework the desired independence is immediate (e.g.~see the logistic and gamma regression examples). 


The derivation of estimation error bounds such as (\ref{eq_defn_Rn}) (as opposed to full sparsistency) usually only requires some kind of local \emph{restricted strong convexity}  (RSC) condition \cite{Negahban2012} on $L_n$.
It is interesting to note that in this paper, it suffices for sparsistency to assume only the LSSC and the positive definiteness of the restricted Hessian at the true parameter. It would be interesting to derive connections between the LSSC and such local RSC conditions, which in turn may shed light on whether the LSSC is necessary to derive sparsistency results, or whether a weaker condition may suffice.


The framework presented here considers general sparse parameters. It is of great theoretical and practical importance to sharpen this framework for structured sparse parameters, e.g., group sparsity, and graphical model learning for networks with bounded degrees.

\appendices

\section{Auxiliary Result for the Non-Structured Case}

In this section, we prove the following claim made in Section 3.  Note that, in contrast to the main definition of the LSSC, the vectors here are \emph{not} necessarily structured.

\begin{prop}
    Consider a function $f \in \mathcal{C}^3 ( \mathrm{\dom{f}} )$ with domain $\dom{f} \subseteq \mathbb{R}^p$. Fix $x^* \in \dom{f}$, and let $\mathcal{N}_{x^*}$ be an open set in $\dom{f}$ containing $x^*$. Let $K \geq 0$. The following statements are equivalent.
    
    \begin{enumerate}
        \item $D^2 f ( x )$ is locally Lipschitz continuous with respect to $x^*$; that is,
        \begin{equation}
        \norm{ D^2 f ( x^* + \delta ) - D^2 f ( x^* ) }_2 \leq K \norm{ \delta }_2, \label{eq_smooth_D2}
        \end{equation}
        for all $\delta \in \mathbb{R}^p$ such that $x^* + \delta \in \mathcal{N}_{x^*}$.
        
        \item $D^3 f ( x )$ is locally bounded; that is, 
        \begin{equation}
        \abs{ D^3 f ( x^* + \delta ) [ u, v, w ] } \leq K \norm{ u }_2 \norm{ v }_2 \norm{ w }_2 \label{eq_smooth_D3}
        \end{equation}
        for all $\delta \in \mathbb{R}^p$ such that $x^* + \delta \in \mathcal{N}_{x^*}$, and for all $u, v, w \in \mathbb{R}^p$.
    \end{enumerate}
    
\end{prop}

\begin{proof}
    Suppose that (\ref{eq_smooth_D2}) holds. By Proposition 3.3, it suffices to prove that
    \begin{equation}
    \abs{ D^3 f ( x^* + \delta ) [ u, u, u ] } \leq K \norm{ u }_2^3 \notag
    \end{equation}
    for all $u \in \mathbb{R}^p$. By definition, we have
    \begin{align}
    \abs{ D^3 f ( x^* + \delta ) [ u, u, u ] }  &= \abs{ \left\langle u, H u \right\rangle } \notag \\
    &\leq \norm{ H }_2 \norm{ u }^2, \notag
    \end{align}
    where
    \begin{equation}
    H := \lim_{t \to 0} \frac{ D^2 f ( x^* + \delta + t u ) - D^2 f ( x^* + \delta ) }{t}. \notag
    \end{equation}
    We therefore have (\ref{eq_smooth_D3}) since $\norm{ H }_2 \leq K \norm{ \delta }_2$ by (\ref{eq_smooth_D2}).
    
    Conversely, suppose that (\ref{eq_smooth_D3}) holds. We have the following Taylor expansion \cite{Zeidler1995}:
    \begin{align}
    D^2 f ( x^* + \delta ) = D^2 f ( x^* ) + \int_0^1 D^3 f ( x_t ) [ \delta ] \, dt, \notag
    \end{align}
    where $x_t := x^* + t \delta$. We also have from (\ref{eq_smooth_D3}) and the definition of the spectral norm that $\norm{ D^3 f ( x^* + \delta ) [ \delta ]}_2 \leq K \norm{u}_2$, and hence
    \begin{align}
    & \norm{ D^2 f ( x^* + \delta ) - D^2 f ( x^* ) }_2 \notag \\
    &\quad = \norm{ \int_0^1 D^3 f ( x_t ) [ \delta ] \, dt }_2 \notag \\
    &\quad \leq K \norm{ \delta }_2. \notag
    \end{align}
    This completes the proof.
\end{proof}

\section{Proof of Theorem 5.1} \label{proof_deterministic}

The proof is based on the optimality conditions on $\hat{\beta}$ for the original problem, and those on $\check{\beta}$ for the restricted problem.  We first observe that $\check{\beta}_n$ exists, since the function $x \mapsto \left\Vert x \right\Vert_1$ is coercive.  We have assumed uniqueness in the theorem statement, thus ensuring the validity of (2).

To achieve sparsistency, it suffices that $\hat{\beta}_n = \check{\beta}_n$ and $\mathrm{supp}\, \check{\beta}_n = \mathrm{supp}\, \beta^*$. We derive sufficient conditions for $\hat{\beta}_n = \check{\beta}_n$ in Lemma \ref{lem_PDW}, and make this sufficient condition explicitly dependent on the problem parameters in Lemma \ref{lem_quasi_support_consistency}. This lemma will require that $\left\Vert \check{\beta}_n - \beta^* \right\Vert_2 \leq R_n$ for some $R_n > 0$.  We will derive an estimation error bound of the form $\left\Vert \check{\beta}_n - \beta^* \right\Vert_2 \leq r_n$ in Lemma \ref{lem_err_bound}. We will then conclude that $\hat{\beta}_n = \check{\beta}_n$ if $r_n \leq R_n$ and the assumptions in Lemma \ref{lem_quasi_support_consistency} are satisfied, from which it will follow that $\mathrm{sign}\, \check{\beta} = \mathrm{sign}\, \beta^*$ provided that $\beta_{\min} \geq r_n$. 

The following lemma is proved via an extension of the techniques of \cite{Wainwright2009}.

\begin{lem} \label{lem_PDW}
We have $\hat{\beta}_n = \check{\beta}_n$ if
\begin{equation}
\left\Vert \left[ \nabla L_n( \check{\beta}_n ) \right]_{\mathcal{S}^{\mathrm{c}}} \right\Vert_{\infty} < \tau_n. \label{eq_original_irrepre}
\end{equation}
\end{lem}
\begin{proof}
Recall that $L_n$ is convex by assumption.  The second assumption of Theorem 5.1 ensures that the restricted optimization problem in $\mathbb{R}^{s}$ is strictly convex, and thus $\check{\beta}_{\mathcal{S}}$ is the only vector the satisfies the corresponding optimality condition:
\begin{equation}
\left[ \nabla L_n ( \check{\beta}_n ) \right]_{\mathcal{S}} + \tau_n \check{z}_{\mathcal{S}} = 0 \label{eq:opt_check}
\end{equation}
for some $\check{z}_{\mathcal{S}}$ such that $\norm{ \check{z}_{\mathcal{S}} }_{\infty} \leq 1$. Moreover, the fact that (\ref{eq_original_irrepre}) is satisfied means that there exists $\check{z}_{\mathcal{S}^\mathrm{c}}$ such that $\norm{ \check{z}_{\mathcal{S}^c} }_{\infty} < 1$ and
\begin{equation}
\nabla L_n( \check{\beta_n} ) + \tau_n \check{z} = 0, \notag
\end{equation}
where $\check{z} := ( \check{z}_{\mathcal{S}}, \check{z}_{\mathcal{S}^c} )$. Therefore, $\check{\beta}_n$ is a minimizer of the original optimization problem in $\mathbb{R}^{p}$.

We now address the uniqueness of $\hat{\beta}$. By a similar argument to Lemma 1 in \cite{Ravikumar2010} (see also Lemma 1(b) in \cite{Wainwright2009}), any minimizer $\tilde{\beta}$ of the original optimization problem satisfies $\tilde{\beta}_{\mathcal{S}^c} = 0$. Thus, since $\check{\beta}$ is the only optimal vector for the restricted optimization problem, we conclude that $\hat{\beta}_n = \check{\beta}_n$ uniquely. 
\end{proof}

We now combine Lemma \ref{lem_PDW} with the assumptions of Theorem 5.1 to obtain the following.

\begin{lem} \label{lem_quasi_support_consistency}
Under assumptions 1, 2, 3 and 6 of Theorem 5.1, we have $\hat{\beta}_n = \check{\beta}_n$ if $\check{\beta} \in \mathcal{N}_{\beta^*} \cap \mathcal{B}_{R_n}$, where $\mathcal{B}_{R_n} := \left\{ \beta: \left\Vert \beta - \beta^* \right\Vert_2 \leq R_n, \beta_{\mathcal{S}^c} = 0, \beta \in \mathbb{R}^p \right\}$ with
\begin{equation}
R_n = \frac{1}{2} \sqrt{\frac{\alpha \tau_n}{K}}. \label{eq:R_n}
\end{equation}
\end{lem}

\begin{proof}
Applying a Taylor expansion at $\beta^*$, and noting that both $\beta^*$ and $\check{\beta}_n$ are supported on $\mathcal{S}$, we obtain
\begin{align}
\left[ \nabla L( \check{\beta}_n ) \right]_{\mathcal{S}^{\mathrm{c}}} &= \left[ \nabla L_n( \beta^* ) \right]_{\mathcal{S}^{\mathrm{c}}} \notag \\
&\quad +\left[ \nabla^2 L_n( \beta^* ) \right]_{\mathcal{S}^{\mathrm{c}}, \mathcal{S}} \left( \check{\beta}_n - \beta^* \right)_{\mathcal{S}} \notag \\
&\quad + \left( \epsilon_n \right)_{\mathcal{S}^{\mathrm{c}}}, \label{eq_taylor_1}
\end{align}
where the remainder term is given by $\epsilon_n = \int_{0} ^{1}(1-t)D^3L_n(\beta_t)[\check{\beta} - \beta^*, \check{\beta} - \beta^*]$dt with $\beta_t := \beta^* + t ( \check{\beta} - \beta^* )$ (see Section 4.5 of \cite{Zeidler1995}), and thus satisfies
\begin{align}
\norm{ \epsilon_n }_{\infty} \leq \sup_{t \in [0,1]} \left\{ \norm{ D^3 L_n ( \beta_t ) [ \check{\beta} - \beta^*, \check{\beta} - \beta^* ] }_{\infty} \right\}. \label{eq:eps_infty} 
\end{align}

Recall the optimality condition for $\check{\beta}$ in (\ref{eq:opt_check}).  Again using a Taylor expansion, we can write this condition as
\begin{align}
\left[ \nabla L_n( \beta^* ) \right]_{\mathcal{S}} + \left[ \nabla^2 L_n( \beta^* ) \right]_{\mathcal{S}, \mathcal{S}} \left( \check{\beta}_n - \beta^* \right)_{\mathcal{S}}  \notag \\
+ (\epsilon_n)_{\mathcal{S}} + \tau_n \check{z}_{\mathcal{S}} = 0. \label{eq_taylor_2}
\end{align}

Recall that $\left[ \nabla^2 L_n( \beta^* ) \right]_{\mathcal{S}, \mathcal{S}}$ is invertible by the second assumption of Theorem 5.1. Solving for $\left( \check{\beta}_n - \beta^* \right)_{\mathcal{S}}$ in (\ref{eq_taylor_2}) and substituting the solution into (\ref{eq_taylor_1}), we obtain
\begin{align}
&\left[ \nabla L_n( \check{\beta}_n ) \right]_{\mathcal{S}^{\mathrm{c}}} \notag \\
&\quad= - \tau_n \left[ \nabla^2 L_n( \beta^* ) \right]_{\mathcal{S}^{\mathrm{c}}, \mathcal{S}} \left[ \nabla^2 L_n( \beta^* ) \right]_{\mathcal{S}, \mathcal{S}}^{-1} \check{z}_{\mathcal{S}} \notag \\
&\quad\quad+\left[ \nabla L( \beta^* ) \right]_{\mathcal{S}^{\mathrm{c}}} \notag \\
&\quad\quad -\left[ \nabla^2 L_n( \beta^* ) \right]_{\mathcal{S}^{\mathrm{c}}, \mathcal{S}} \left[ \nabla^2 L_n( \beta^* ) \right]_{\mathcal{S}, \mathcal{S}}^{-1} \left[ \nabla L_n( \beta^* ) \right]_{\mathcal{S}} \notag \\
&\quad\quad+( \epsilon_n )_{\mathcal{S}^{\mathrm{c}}} \notag \\
&\quad\quad-\left[ \nabla^2 L_n( \beta^* ) \right]_{\mathcal{S}^{\mathrm{c}}, \mathcal{S}} \left[ \nabla^2 L_n( \beta^* ) \right]_{\mathcal{S}, \mathcal{S}}^{-1} ( \epsilon_n )_{\mathcal{S}}. \notag
\end{align}
Using the irrepresentability condition (assumption 3 of Theorem 5.1) and the triangle inequality, we have $\left\Vert \left[ \nabla L_n( \check{\beta}_n ) \right]_{\mathcal{S}^{\mathrm{c}}} \right\Vert_{\infty} < \tau_n$ provided that
\begin{equation}
\max \left\{ \left\Vert \nabla L_n( \beta^* ) \right\Vert_{\infty}, \left\Vert \epsilon_n \right\Vert_{\infty} \right\} \leq \frac{\alpha}{4} \tau_n. \notag
\end{equation}
The first requirement $\left\Vert \nabla L_n( \beta^* ) \right\Vert_{\infty} \leq ( \alpha / 4 ) \tau_n$ is simply assumption 6 of Theorem 5.1, so it remains to determine a sufficient condition for $\left\Vert \epsilon_n \right\Vert_{\infty} \leq ( \alpha / 4 ) \tau_n$. Since $L_n$ satisfies the $( \beta^* , \mathcal{N}_{\beta^*} )$-LSSC with parameter $K$, we have from \eqref{eq:eps_infty} that
\begin{equation}
\left\Vert \epsilon_n \right\Vert_{\infty} \leq K \left\Vert \check{\beta} - \beta^* \right\Vert_{2}^2, \notag
\end{equation}
provided that $\check{\beta} \in \mathcal{N}_{\beta^*}$ (since $\mathcal{N}_{\beta^*}$ is convex by assumption, this implies $\beta_t \in \mathcal{N}_{\beta^*}$). Thus, to have $\left\Vert \epsilon_n \right\Vert_{\infty} \leq \frac{\alpha}{4} \tau_n$, it suffices that
\begin{equation}
\left\Vert \check{\beta} - \beta^* \right\Vert_{2} \leq \frac{1}{2} \sqrt{ \frac{\alpha \tau_n }{K} } \notag
\end{equation}
and $\check{\beta} \in \mathcal{N}_{\beta^*}$.
\end{proof}

To bound the distance $\left\Vert \check{\beta} - \beta^* \right\Vert_2$, we adopt an approach from \cite{Ravikumar2010,Rothman2008}.  We begin with an auxiliary lemma.

\begin{lem} \label{lem:boundary_result}
Let $g: \mathbb{R}^p \to \mathbb{R}$ be a convex function, and let $z \in \mathbb{R}^p$ be such that $g( z ) \leq 0$. Let $\mathcal{B} \subset \mathbb{R}^p$ be a closed set, and let $\partial \mathcal{B}$ be its boundary. If $g > 0$ on $\partial \mathcal{B}$ and $g( b ) \leq 0$ for some $b \in \mathcal{B} \setminus \partial \mathcal{B}$, then $x \in \mathcal{B}$.
\end{lem}
\begin{proof}
We use a proof by contradiction. Suppose that $z \notin \mathcal{B}$. We first note that there exists some $t^* \in ( 0, 1 )$ such that $b + t^* ( z - b ) \in \partial \mathcal{B}$; if such a $t^*$ did not exist, then we would have $z_{t} := b + t( z - b ) \to z$ as $t \to 1$, which is impossible since $z \notin \mathcal{B}$ and $\mathcal{B}$ is closed.

We now use the convexity of $g$ to write
\begin{equation}
g( b + t^* ( x - b ) ) \leq ( 1 - t^* ) g( b ) + t^* g( x ) \leq 0, \notag
\end{equation}
which is a contradiction since $g>0$ on $\partial\mathcal{B}$.
\end{proof}

The following lemma presents the desired bound on $\left\Vert \check{\beta}_n - \beta^* \right\Vert_2$; note that this can be interpreted as the estimation error in the $n > p$ setting, considering $\beta^*_{\mathcal{S}}$ as the parameter to be estimated.

\begin{lem} \label{lem_err_bound}
Define the set
\begin{equation}
\mathcal{B}_{r_n} := \left\{ \beta \in \mathbb{R}^p : \left\Vert \beta - \beta^* \right\Vert_{2} \leq r_n, \beta_{\mathcal{S}^c} = 0 \right\}, \notag
\end{equation}
where
\begin{equation}
r_n := \frac{\alpha + 4 }{\lambda_{\min}} \sqrt{s}\tau_n. \label{eq:r_n}
\end{equation}
Under assumptions 1, 2, 6 and 7 of Theorem 5.1, if
\begin{equation}
\tau_n < \frac{3 \lambda_{\min}^2}{2 ( \alpha + 4 )  sK }, \label{eq:tau_cond1}
\end{equation}
then $\check{ \beta }_n \in \mathcal{B}_{r_n}$. 
\end{lem}
\begin{proof}
Set $s = \left\vert \mathcal{S} \right\vert$, and for $\beta\in\mathbb{R}^s$ let $Z(\beta) = ( \beta, 0 ) \in \mathbb{R}^p$ be the zero-padding mapping, where $( \beta, 0 )$ denotes the vector that equals to $\beta$ on $\mathcal{S}$ and $0$ on $\mathcal{S}^{\mathrm{c}}$. Then we have
\begin{equation}
\check{\beta}_{\mathcal{S}} = \arg \min_{\beta \in \mathbb{R}^s} \left\{ ( L_n \circ Z )( \beta ) + \tau_n \left\Vert \beta \right\Vert_1 \right\}. \notag
\end{equation}
For $\delta \in \mathbb{R}^s$, define 
\begin{multline}
g(\delta) =( L_n \circ Z )( \beta_{\mathcal{S}}^* + \delta ) - ( L_n \circ Z )( \beta_{\mathcal{S}}^* ) + \\
\tau_n \left( \left\Vert \beta_{\mathcal{S}}^* + \delta \right\Vert_1 - \left\Vert \beta_{\mathcal{S}}^* \right\Vert_1 \right). \notag
\end{multline}
We trivially have $g(0) = 0$, and thus $g( \delta^* ) \leq g(0) = 0$, where $\delta^* := \check{\beta}_{\mathcal{S}} - \beta^*_{\mathcal{S}}$. Now our goal is prove that $g > 0$ on the boundary of $(\mathcal{B}_{r_n})_{\mathcal{S}} := \left\{ \delta \in \mathbb{R}^s : \left\Vert \delta \right\Vert_{2} \leq r_n \right\}$, thus permitting the application of Lemma \ref{lem:boundary_result}.

We proceed by deriving a lower bound on $g( \delta )$. We define $\phi(t) := ( L_n \circ Z )( \beta_\mathcal{S}^* + t \delta )$, and write the following Taylor expansion:
\begin{align}
&( L_n \circ Z )( \beta^*_{\mathcal{S}} + \delta ) - ( L_n \circ Z )( \beta^*_{\mathcal{S}} ) \notag \\
&\quad= \phi(1) - \phi(0) \notag \\
&\quad= \phi'(0) + \frac{1}{2} \phi''(0) + \frac{1}{6} \phi'''( \tilde{t} ), \notag
\end{align}
for some $\tilde{t} \in [ 0, 1 ]$ (recall that $L_n$ is three times differentiable by assumption). We bound the term $\phi'(0)$ as follows: 
\begin{align}
\left\vert \phi'(0) \right\vert &= \left\vert \left\langle \left[ \nabla L_n( \beta^* ) \right]_{\mathcal{S}},\delta  \right\rangle \right\vert \notag \\
&\leq \sqrt{s} \left\Vert \left[ \nabla L_n( \beta^* ) \right]_{\mathcal{S}} \right\Vert_\infty \left\Vert \delta \right\Vert_{2} \notag \\
&\leq \frac{\alpha \tau_n}{4} \sqrt{s} \left\Vert \delta \right\Vert_{2}, \notag
\end{align}
where the first step is by H\"{o}lder's inequality and the identity $\|z\|_{2} \le \sqrt{s}\|z\|_{1}$, and the second step uses assumption 6 of Theorem 5.1. To bound the term $\phi''(0)$, we use the second assumption of Theorem 5.1 to write
\begin{equation}
\phi '' ( 0 ) = \delta^T \left[ \nabla^2 L_n( \beta^* ) \right]_{\mathcal{S}, \mathcal{S}} \delta \geq \lambda_{\min} \left\Vert \delta \right\Vert_{2}^2. \notag
\end{equation}
We now turn to the term $\phi'''( \tilde{t} )$.  Again using the fact that $L_n$ satisfies the $(\beta^*,\mathcal{N}_{\beta^*})$-LSSC with parameter $K$, it immediately follows that $( L_n \circ Z )$ satisfies the $( \beta^*_{\mathcal{S}}, \left( \mathcal{N}_{\beta^*} \right)_{\mathcal{S}} )$-LSSC with parameter $K$, where $\left( \mathcal{N}_{\beta} \right)_{\mathcal{S}} = \left\{ \beta_{\mathcal{S}} : \beta \in \mathcal{N}_{\beta^*} \right\}$. Hence, and also making use of H\"older's inequality and the fact that $\|z\|_1 \le \sqrt{s}\|z\|_2$ ($z \in \mathbb{R}^s$), we have
\begin{align}
\left\vert \phi'''( \tilde{t} ) \right\vert &= \left\vert D^3( L_n \circ Z ) ( \beta_{\mathcal{S}}^* + \tilde{t} \delta ) [ \delta, \delta, \delta ] \right\vert \notag \\
& \leq \norm{ \delta }_1 \norm{ D^3( L_n \circ Z ) ( \beta_{\mathcal{S}}^* + \tilde{t} \delta ) [ \delta, \delta ] }_{\infty} \notag \\
& \leq K \sqrt{s}\left\Vert \delta \right\Vert_{2}^3\notag
\end{align}
provided that $\beta_{\mathcal{S}}^* + \tilde{t} \delta \in \left( \mathcal{N}_{\beta} \right)_{\mathcal{S}}$.  Since $\mathcal{B}_{r_n} \subseteq \mathcal{N}_{\beta^*}$ by assumption 7 of Theorem 5.1, the latter condition holds provided that $\delta \in (\mathcal{B}_{r_n})_{\mathcal{S}} $.

Using the triangle inequality, we have
\begin{equation}
\left\vert \left\Vert \beta_{\mathcal{S}}^* + \delta \right\Vert_1 - \left\Vert \beta_{\mathcal{S}}^*  \right\Vert_1 \right\vert \leq \left\Vert \delta \right\Vert_1 \leq \sqrt{s} \left\Vert \delta \right\Vert_{2}. \notag
\end{equation}
Hence, and combining the preceding bounds, we have $g( \delta ) \geq f \left( \left\Vert \delta \right\Vert_2 \right)$, where
\begin{align}
f(x) = - \frac{\alpha \tau_n}{4} \sqrt{s} x + \frac{\lambda_{\min}}{2} x^2 - \frac{K\sqrt{s}}{6} x^3 - \sqrt{s} \tau_n x. \notag
\end{align}
Observe that if the inequality 
\begin{equation}
0 < x < \frac{3 \lambda_{\min}}{2 K \sqrt{s}}. \label{eq_condition_fx}
\end{equation}
holds, then we can bound the coefficient to $x^3$ in terms of that of $x^2$ to obtain 
\begin{equation}
f ( x ) > \frac{\lambda_{\min}}{4} x^2 - \left( 1 + \frac{\alpha}{4} \right) \sqrt{s} \tau_n x. \label{eq_fx}
\end{equation}
By a direct calculation, this lower bound has roots at $0$ and $r_n$ (see \eqref{eq:r_n}), and hence $f ( r_n ) > 0$ provided that $x=r_n$ satisfies (\ref{eq_condition_fx}). By a direct substitution, this condition can be ensured by requiring that
\begin{equation}
\tau_n < \frac{3 \lambda_{\min}^2}{2 ( \alpha + 4 ) Ks }. 
\end{equation}
Recalling that $g( \delta ) \geq f \left( \left\Vert \delta \right\Vert_2 \right)$, we have proved that $g$ satisfies the conditions of Lemma \ref{lem:boundary_result} with $z = \delta^*$, $b=0$, and $\mathcal{B} = (\mathcal{B}_{r_n})_{\mathcal{S}}$, and we thus have $\delta^* \in (\mathcal{B}_{r_n})_{\mathcal{S}}$, or equivalently $\check{ \beta }_n \in \mathcal{B}_{r_n}$. 



\end{proof}

We now combine the preceding lemmas to obtain Theorem 5.1. We require $r_n \leq R_n$ so the assumption that $\left\Vert \check{\beta} - \beta^* \right\Vert_{\infty} \leq R_n$ in Lemma \ref{lem_quasi_support_consistency} is satisfied. From the definitions in \eqref{eq:R_n} and \eqref{eq:r_n}, this is equivalent to requiring
\begin{equation}
\tau_n \leq \frac{\lambda_{\min}^2}{4 \left( \alpha + 4 \right)^2} \frac{\alpha}{K s}, \notag
\end{equation}
which is true by assumption 5 of the theorem.  This assumption also implies that (\ref{eq:tau_cond1}) holds, since $\frac{\alpha}{4(\alpha+4)} \le \frac{3}{2}$ for any $\alpha\ge0$. Finally, by the conclusion of Lemma \ref{lem_err_bound}, we have successful sign pattern recovery if $\beta_{\min} \geq r_n$, thus recovering assumption 4 of the theorem.

\section{Proofs of the Results in Section 6} \label{proof_appl}
\subsection{Proof of Corollary 6.2}
By a direct differentiation, we obtain for $j \in \{ 1, \ldots, p \}$ that
\begin{equation}
\left[ \nabla L_n ( \beta^* ) \right]_j = - \sum_{i = 1}^n \varepsilon_i ( x_i )_j, \notag
\end{equation}
where $\varepsilon_i = n^{-1} \left( Y_i - \mathsf{E}\, Y_i \right)$. 

Fix $j \in \{ 1, \ldots, p \}$, and let $X_i := n^{-1} ( x_i )_j Y_i$. As $X_1, \ldots, X_n$ are bounded, they can be characterized using Hoeffding's inequality \cite{Boucheron2013}.

\begin{thm}[Hoeffding's Inequality]
Let $X_1, \ldots, X_n$ be independent random variables such that $X_i$ takes its value in $[ a_i, b_i ]$ almost surely for all $i \in \{ 1, \ldots, n \}$. Then
\begin{align}
&\mathsf{P}\, \left\{ \left\vert \sum_{i = 1}^n \left( X_i - \mathsf{E}\, X_i \right) \right\vert \geq t \right\} \notag \\
&\quad\leq 2 \exp \left[ - \frac{2 t^2}{ \sum_{i = 1}^n ( b_i - a_i )^2 } \right]. \notag
\end{align}
\end{thm}

In our case, we can set $( b_i - a_i )^2 = n^{-2} ( x_i )_j^2$, since $Y_i \in \{0,1\}$.  Since $\sum_{i=1}^{n} |(x_i)_j|^2 \le n$ for all $k$ by assumption, we obtain
\begin{equation}
\sum_{i = 1}^n ( b_i - a_i )^2 \leq \frac{1}{n}.
\end{equation}
Thus, by Hoeffding's inequality and the union bound, we obtain
\begin{align}
&\mathsf{P}\, \left\{ \left\Vert \nabla L_n ( \beta^* ) \right\Vert_{\infty} \geq \frac{\alpha \tau_n}{4} \right\} \notag \\
& \quad\leq \sum_{j = 1}^p \mathsf{P}\, \left\{ \left\vert \left[ \nabla L_n ( \beta^* ) \right]_j \right\vert \geq \frac{\alpha \tau_n}{4} \right\} \notag \\
& \quad\leq \left. 2 \exp \left( \ln p - 2 n t^2 \right) \right\vert_{t = \frac{\alpha \tau_n}{4}}. \notag
\end{align}

This decays to zero provided that $\tau_n \gg ( n^{-1} \log p )^{1/2}$.  Substituting this scaling into the fifth condition of Theorem 5.1, we obtain the condition $s^2 \left( \log p \right) \nu_n^4\gamma_n^2 \ll n$.  The required uniqueness of $\check{\beta}$ can be proved by showing that  the composition $L_n \circ Z$ (with $Z$ being the zero-padding of a vector in $\mathbb{R}^s$) is strictly convex, given the second condition of Theorem 5.1. One way to prove this is via self-concordant like inequalities \cite{Tran-Dinh2013}; we omit the proof here for brevity.

\subsection{Proof of Corollary 6.3}
Let $Y_1, \ldots, Y_n$ be independent gamma random variables with shape parameter $k > 0$ and scale parameter $\theta_i$ respectively. We have, for $q \in \mathbb{N}$,
\begin{equation}
\mathsf{E}\, \left\vert Y_i \right\vert^q = \frac{\Gamma ( q + k )}{\Gamma( k )} \theta_i^q, \notag
\end{equation}
where $\Gamma$ denotes the gamma function.

To study the concentration of measure behavior of $\nabla L_n ( \beta^* )$, we use the following result \cite{Boucheron2013}.

\begin{thm}[Bernstein's Inequality]
Let $X_1, \ldots, X_n$ be independent real random variables. Suppose that there exist $v > 0$ and $c > 0$ such that $\sum_{i = 1}^n \mathsf{E}\, X_i^2 \leq v$, and
\begin{equation}
\sum_{i = 1}^n \mathsf{E}\,\left\vert X_i \right\vert^q \leq \frac{q!}{2} v c^{q - 2} \notag
\end{equation} 
for all integers $q \geq 3$. Then
\begin{equation}
\mathsf{P}\, \left\{ \left\vert \sum_{i = 1}^n \left( X_i - \mathsf{E}\, X_i \right) \right\vert \geq t \right\} \leq 2 \exp \left[ - \frac{t^2}{2 ( v + c t )} \right]. \notag
\end{equation}
\end{thm}

We proceed by evaluating the required moments for our setting.  By a direct differentiation, we obtain
\begin{equation}
\left[ \nabla L_n( \beta^* ) \right]_j = \sum_{i = 1}^n \varepsilon_i \left( x_i \right)_j \notag
\end{equation}
for $j \in \{ 1, \ldots, p \}$, where $\varepsilon_i := n^{-1} \left( Y_i - \mathsf{E}\, Y_i \right)$. 

Fix $j \in \{ 1, \ldots, p \}$, and let $X_i := n^{-1} ( x_i )_j Y_i$. We have
\begin{align}
\sum_{i = 1}^n \mathsf{E}\, X_i^2 &= \sum_{i = 1}^n \frac{( x_i )_j^2}{n^2} \mathsf{E}\, Y_i^2 \notag \\
& = \sum_{i = 1}^n \frac{( x_i )_j^2}{n^2} \frac{\Gamma ( k + 2 )}{\Gamma( k )} \theta_i^2. \notag
\end{align}
Recall that $\theta_i = k^{-1} \left\langle x_i, \beta^* \right\rangle^{-1}$. Using the first displayed equation in Section 7.3, we have
\begin{equation}
\theta_i \leq \left( k \mu_n\right)^{-1}, \label{eq_upper_bound_thetai}
\end{equation}
and thus
\begin{align}
\sum_{i = 1}^n \mathsf{E}\, X_i^2 &\leq \frac{1}{( n \mu_n )^2} \frac{\Gamma ( k + 2 )}{ k^2 \Gamma( k ) } \sum_{i = 1}^n \frac{ ( x_i )_j^2 }{ \left\Vert x_i \right\Vert_2^2 } \notag \\
& \leq \frac{1}{n \mu_n^2} \frac{\Gamma ( k + 2 )}{k^2 \Gamma ( k )}, \notag
\end{align}
where we have applied the assumption $\sum_{i = 1}^n ( x_i )_j^2 \le n$.  Using the identity $\Gamma ( k + 2 ) = k ( k + 1 ) \Gamma ( k )$, we obtain
\begin{equation}
\sum_{i = 1}^n \mathsf{E}\, X_i^2 \leq \frac{k + 1}{n \mu_n^2 k}. \notag
\end{equation}
As for the moments of higher orders, we have
\begin{align}
\sum_{i = 1}^n \mathsf{E}\, \left\vert X_i \right\vert^q &= \sum_{i = 1}^n \frac{\left\vert ( x_i )_j \right\vert^q}{n^q} \mathsf{E}\, \left\vert Y_i \right\vert^q \notag \\
& = \sum_{i = 1}^n \frac{\left\vert ( x_i )_j \right\vert^q}{n^q} \frac{\Gamma ( k + q )}{\Gamma ( k )} \theta_i^q. \notag
\end{align}
With the upper bound (\ref{eq_upper_bound_thetai}) on $\theta_i$, we have
\begin{align}
\sum_{i = 1}^n \mathsf{E}\, \left\vert X_i \right\vert^q &\leq \frac{\Gamma(k+q)}{ (kn\mu_n)^q \Gamma(k) } \sum_{i = 1}^n \left\vert ( x_i )_j \right\vert^q \notag \\
& = \frac{\Gamma(k+q)}{ (kn\mu_n)^q \Gamma(k) } \left\Vert ( ( x_1 )_j, \ldots, ( x_n )_j ) \right\Vert_q^q. \notag
\end{align}
Using the identity $\left\Vert z \right\Vert_q \leq \left\Vert z \right\Vert_2$ for $q \geq 2$, and the assumption $\sum_{i = 1}^n ( x_i )_j^2 \le n$, we obtain
\begin{equation}
\sum_{i = 1}^n \mathsf{E}\, \left\vert X_i \right\vert^q \leq \frac{\Gamma(k+q)}{(k\sqrt{n}\mu_n)^q \Gamma(k)}. \notag
\end{equation}

For $k \in ( 0, 1]$, we have $\frac{\Gamma(k+q)}{\Gamma(q)} \le q!$, and hence by a direct substitution it suffices to choose
\begin{equation}
v = \frac{k + 1}{ n \mu_n^2 k^2}, \quad c = \frac{1}{k\sqrt{n} \mu_n}. \label{eq_k_small}
\end{equation}
For $k \in (1, \infty)$, we have by induction on $q$ that $\frac{\Gamma(k+q)}{\Gamma(q)} \leq q! k^q$.  Thus, for $k \in ( 1, \infty )$, it suffices that
\begin{equation}
v = \frac{2 k}{n \mu_n^2}, \quad c = \frac{1}{\sqrt{n} \mu_n}. \label{eq_k_large}
\end{equation}

Thus, applying Bernstein's inequality and the union bound, we obtain
\begin{align}
& \mathsf{P}\, \left\{ \left\Vert \nabla L_n ( \beta^* ) \right\Vert_{\infty}  \geq \frac{\alpha \tau_n}{4} \right\} \notag \\
& \quad\leq \sum_{i = 1}^p \mathsf{P}\, \left\{ \left\vert \left[ \nabla L_n ( \beta^* ) \right]_i \right\vert \geq \frac{\alpha \tau_n}{4} \right\} \notag \\
& \quad\leq \left. 2 \exp \left[ \ln p - \frac{t^2}{2 ( v + c t )} \right] \right\vert_{t = \frac{\alpha \tau_n}{4}}. \notag
\end{align}

Since $L_n$ is self-concordant and $\left[ D^2 L_n (\beta^*) \right]_{\mathcal{S}, \mathcal{S}}$ is positive definite by assumption, the composition $L_n \circ Z$ with the padding operator $Z$ is strictly convex \cite{Nesterov2004,Nesterov1994} and thus $\check{\beta}_n$ uniquely exists. Therefore, we can apply Theorem 5.1.  The scaling laws on $\tau_n$ and $(p,n,s)$ follow via the same argument to that in the proof of Corollary 6.2.  Note that the final condition of Theorem 5.1 also imposes conditions on $(p,n,s)$, but for this term even the weaker condition $s^2(\log p)\nu_n^2 \ll n$ suffices.

\section{Proof of Corollary 6.4}

By a direct differentiation, we obtain
\begin{equation}
\nabla L_n ( \Theta^* ) = \hat{\Sigma}_n - \left( \Theta^* \right)^{-1} = \hat{\Sigma}_n - \Sigma. \notag
\end{equation}

We apply the following lemma from \cite{Ravikumar2011} to study the concentration behavior of $\nabla L_n( \Theta^* )$. 

\begin{lem}
Let $\Sigma$ and $\hat{\Sigma}_n$ be defined as in Section 6.4. We have
\begin{align}
&\mathsf{P}\, \left\{ \left\vert \left( \hat{\Sigma}_n \right)_{i,j} - \Sigma_{i,j} \right\vert > t \right\} \notag \\
& \quad\leq 4 \exp \left[ - \frac{ n t^2 }{ 128 ( 1 + 4 c^2 )^2 \kappa_{\Sigma^*}^2 } \right], \notag
\end{align}
for all $t \in ( 0, 8 \kappa_{\Sigma^*} ( 1 + c )^2 )$.
\end{lem}

Using the union bound, we have
\begin{align}
& \mathsf{P}\, \left\{ \left\Vert \nabla L_n ( \Theta^* ) \right\Vert_{\infty} \leq \frac{\alpha \tau_n}{4} \right\} \notag \\
& \quad\leq 4 p^2 \left. \exp \left[ - \frac{ n t^2 }{ 128 ( 1 + 4 \sigma^2 )^2 \kappa_{\Sigma^*}^2 } \right] \right\vert_{t = \frac{\alpha \tau_n}{4}}, \notag
\end{align}
provided that $\tau_n \to 0$, and that $n$ is large enough so that the upper bound on $t$ in the lemma is satisfied.

Define
\begin{align}
\check{\Theta}_n \in \arg\min_{\Theta} \left\{ L_n ( \Theta ) + \tau_n \left\vert \Theta \right\vert_1: \right. \quad\notag \\
\left. \Theta > 0, \Theta_{\mathcal{S}^c} = 0, \Theta \in \mathrm{R}^{p \times p} \right\}.
\end{align}
Since $L_n$ is self-concordant and $\left[ D^2 L_n (\Theta^*) \right]_{\mathcal{S}, \mathcal{S}}$ is positive definite by assumption, the composition $L_n \circ Z$ with the padding operator $Z$ is strictly convex \cite{Nesterov2004,Nesterov1994} and thus $\check{\Theta}_n$ uniquely exists. Therefore, we can apply Theorem 5.1.  The scaling laws on $\tau_n$ and $(p,n,s)$ follow via the same arguments as the preceding examples.

\bibliographystyle{IEEEtranS}
\bibliography{list,sml}

%
%
%
%

\end{document}